\documentclass[11pt]{article}
\usepackage{url}
\usepackage{latexsym,color,amsmath,amsthm,amssymb,amscd,amsfonts,inputenc,bbm}
\usepackage{graphicx} 
\usepackage{scalefnt}
\usepackage[font=small,labelfont=bf]{caption}
\usepackage{enumitem}

\usepackage{amsopn}
\usepackage{relsize}
\usepackage{hyperref}
\usepackage{mathrsfs}
\usepackage{mathtools}
\usepackage[dvipsnames]{xcolor}
\usepackage{fancyhdr}

\newtheoremstyle{nonum}{}{}{\itshape}{}{\bfseries}{.}{ }{\thmnote{#3}}

\newtheorem{thm}{Theorem}
\newtheorem*{thm*}{Theorem}

\newtheorem*{cor*}{Corollary}
\newtheorem{lem}[thm]{Lemma}
\newtheorem*{lem*}{Lemma}
\newtheorem{prop}[thm]{Proposition}
\newtheorem{rem}[thm]{Remark}
\newtheorem*{rem*}{Remark}
\newtheorem{conj}[thm]{Conjecture}
\newtheorem*{conj*}{Conjecture}
\newtheorem{exm}[thm]{Example}

\newtheorem*{definition*}{Definition}

\newtheorem*{fact*}{Fact}

\newtheorem*{rems*}{Remarks}
\newtheorem{rems}[thm]{Remarks}

\theoremstyle{nonum}

\newcommand{\RR}{\mathbb R}

\newcommand{\N}{\mathbb N}

\def\sp{{\rm sp }}

\newcommand{\iprod}[2]{\langle #1,#2 \rangle} 

\newcommand{\vol}{{\rm{Vol}}}

\def\conv{{\rm conv}}

\makeatletter
\def\moverlay{\mathpalette\mov@rlay}
\def\mov@rlay#1#2{\leavevmode\vtop{%
		\baselineskip\z@skip \lineskiplimit-\maxdimen
		\ialign{\hfil$\m@th#1##$\hfil\cr#2\crcr}}}
\newcommand{\charfusion}[3][\mathord]{
	#1{\ifx#1\mathop\vphantom{#2}\fiי
		\mathpalette\mov@rlay{#2\cr#3}
	}
	\ifx#1\mathop\expandafter\displaylimits\fi}
\makeatother

\DeclarePairedDelimiter{\parens}()

\newcommand{\setdef}[2]{\left\{ #1 \ \middle| \ #2 \right\}}
\newcommand{\diff}{\mathop{}\!d}

\begin{document}
	\title{A refinement of the \v{S}id\'ak-Khatri inequality and a strong Gaussian correlation conjecture}
 \author{Rotem Assouline, Arnon Chor, and Shay Sadovsky}
	\maketitle

\begin{abstract}
     We prove two special cases of a strengthened Gaussian correlation conjecture, due to Tehranchi, and show that if the conjecture holds asymptotically, it holds for any dimension. Additionally, we use these special cases to prove a refined version of the \v{S}id\'ak-Khatri inequality.
\end{abstract}
\section{Introduction}
Over 50 years ago, \v{S}id\'ak \cite{sidak1967rectangular,sidak1968multivariate} and independently, Khatri \cite{khatri1967certain}, proved the following inequality for $X_1,\dots, X_n$ symmetric Gaussian random variables:
\[
    \Pr(|X_i|\le 1 \ \forall i \in [n]) \ge \prod_{i=1}^n \Pr(|X_i|\le 1) .
\]
Our first theorem is a refinement of this result.
\begin{thm}[Refined \v{S}id\'ak-Khatri inequality]\label{thm:refined-sidak-prob}
    Let $X_1,\dots, X_n$ be zero-mean jointly Gaussian real random variables, and let $a_1 \in (0,\infty]$. Then,
    \begin{align}\label{eq:simple-sidak}
        \begin{split}
            &\Pr(|X_1|\le 1+a_1)\Pr(|X_i|\le 1 \ \forall i \in [n])\\
            &\quad \ge \Pr(|X_1|\le 1) \Pr(|X_1|\le 1+a_1,|X_2|\le 1,\dots, |X_n|\le 1  ).
        \end{split}
    \end{align}
    In particular, we find that
    \begin{equation}\label{eq:refined-sidak}
            \dfrac{\Pr(|X_i|\le 1 \ \forall i \in [n])}{ \prod_{i=1}^n \Pr(|X_i|\le 1) } \ge \sup_{a_1,\dots,a_n\in (0,\infty) } \dfrac{\Pr(|X_i|\le 1+a_i \ \forall i \in [n])}{ \prod_{i=1}^n \Pr(|X_i|\le 1+a_i) }. 
    \end{equation}
\end{thm}
Clearly, when $a_i = \infty$ for all $i \in [n]$, the inequality becomes the \v{S}id\'ak-Kahtri inequality. In addition, renormalizing, we find that 
\[
    \dfrac{\Pr(|X_i|\le c_i \ \forall i \in [n])}{ \prod_{i=1}^n \Pr(|X_i|\le c_i) }
\]
is increasing each of the $c_i$'s (see Remark \ref{rem:confidence-interval}).

The Gaussian Correlation inequality, which is a generalization of the \v{S}id\'ak-Khatri inequality, was proved by Royen in \cite{royen2014simple} and states the following:
\begin{equation}\label{eq:prob-GCI}
    \Pr(|X_i|\le 1 \ \forall i \in [n])\ge \Pr(|X_i|\le 1 \ \forall i \in [k])\Pr(|X_i|\le 1 \ \forall i \in [n] \setminus [k]).
\end{equation}

In attempt to generalize Theorem \ref{thm:refined-sidak-prob}, we present the following conjecture, which is equivalent to a conjecture of Teheranchi \cite{Tehranchi} (stated below).
\begin{conj}\label{conj:strong-GCI-prob}
    Let $X_1,\dots, X_n$ be zero-mean jointly Gaussian real random variables, and let $s_1,\dots, s_n, t_1,\dots, t_n \in [0,\infty)\cup \{\infty\}$. Then,
    \begin{equation} \label{eq:strong-GCI-prob}
    \begin{aligned}
        \Pr(|X_i|\le s_i+t_i \  \forall i\in [n])&\Pr(|X_i|\le \min\{s_i,t_i\} \  \forall i\in [n])\\
        \ge \Pr(|X_i|\le s_i \  \forall i\in [n])&\Pr(|X_i|\le t_i \ \forall i\in [n]).
    \end{aligned}
    \end{equation}
\end{conj}

Our proof relies on the geometric interpretation of the inequality with respect to the usual Gaussian probability measure in $\RR^d$, denoted henceforth as $\gamma_d$. The classical \v{S}id\'ak-Khatri inequality can be interperted as the inequality $\gamma_d(K \cap T) \ge \gamma_d(K)\gamma_d(T)$ when $K$ is an origin-symmetric convex set and $T$ is a symmetric slab. In this geometric setting, this is due to Giannopoulos \cite{giannopoulos1997vector-ballancing} and Szarek and Werner \cite{szarek-werner1999nonsymmetric}.

Theorem \ref{thm:refined-sidak-prob} is restated and proved in the geometric setting as follows.
\begin{thm}\label{thm:slab}
    Let $K\subseteq \RR^d$ be an origin-symmetric convex body and $T=\{x:\ |\iprod{x}{u}|\le 1\}\subset \RR^d$ a symmetric slab with $u\in \RR^d$. Then,
    \begin{equation}\label{eq:geometric-sidak}
        \gamma_d(\conv\{K\cup T\})\gamma_d(K\cap T) \ge \gamma_d(K)\gamma_d(T).
        \end{equation}
\end{thm}

Equation \eqref{eq:geometric-sidak} does not hold for all origin symmetric and convex $K,T \subseteq \RR^d$, as will be seen in Example \ref{ex:not-for-conv}. However, the following weaker inequality, which is a strengthening of Royen's Gaussian correlation inequality, was conjectured to be true by Tehranchi \cite{Tehranchi}.
\begin{conj}\label{conj:strong-GCI}
	Let $d>0$ and let $K,T\subseteq \RR^d$ be origin-symmetric convex sets, then
	\begin{equation}\label{eq:strong-GCI}
	\gamma_d(K+T)\gamma_d(K\cap T)\ge \gamma_d(K)\gamma_d(T)
	\end{equation}
	where $\gamma$ is the usual Gaussian probability measure.
\end{conj}
In~\cite{Tehranchi}, Tehranchi showed that Conjecture \ref{conj:strong-GCI} holds up to some constant, improving on a result by Schechetman, Schlumprecht, and Zinn \cite{Schechtman1998gaussian} (these results and other related works are elaborated on in the Background section).

The proof of the equivalence of Conjectures \ref{conj:strong-GCI} and \ref{conj:strong-GCI-prob} is a classical argument, which we present in Proposition \ref{prop:equivalent-formulations} in the Background section. Using a similar argument we show that Theorem \ref{thm:slab} implies Theorem \ref{thm:refined-sidak-prob}.

We also present two partial results in the direction of the conjecture. First, we prove it for unconditional convex sets.
\begin{thm}\label{thm:uncond-strong-GCI}
		Let $K,T\subseteq \RR^d$ be unconditional convex sets, then \eqref{eq:strong-GCI} holds.
\end{thm}
Second, we show that an asymptotic version of Conjecture \ref{conj:strong-GCI} would imply the conjecture itself.
\begin{thm}\label{thm:asymptotic-GCI}
    Assume that there is a sequence of positive numbers $(c_N)$ with $\lim_{N\to \infty} c_N^{1/N} =1$, such that for any $N \in \N$ and $K,T$ origin-symmetric convex sets in $\RR^N$,
    \[\gamma_N(K+T)\gamma_N(K\cap T)\ge c_N \gamma_N(K)\gamma_N(T).\]
    Then, Conjecture \ref{conj:strong-GCI} holds for all $d>0$ and all $K,T \subseteq \RR^d$.
\end{thm}

\subsection{Acknowledgements}
The authors would like to thank Ruth Heller and Yosef Rinott for their useful comments and discussion. The third named author would like to thank Matthieu Fradelizi for introducing her to Tehranchi's conjecture, and is thankful to the Azrieli foundation for the award of an Azrieli fellowship. The first author is supported by a grant from the Israel Science Foundation (ISF). The second and third named authors were supported in part by the European Research Council (ERC) under the European Union’s Horizon 2020  research and innovation programme (grant agreement No 770127).
\section{Background}
In this note, we switch between a probabilistic setting of $n$ random variables and a geometric setting of functions and sets in $\RR^d$. We will use some capital letters (e.g. $X,Y,X_i, Y_j$) to denote random vectors and some to denote convex sets (e.g. $K,L,T$). The inner product $\iprod{\cdot}{\cdot}$ will be the standard product in $\RR^d$. 
The Gaussian measure in $\RR^d$, denoted by $\gamma_d$, is the probability measure with density proportional to $e^{-\|x\|^2/2}$.
Finally, the Minkowski sum of two sets is the set 
$K+T := \{x+y: x\in K,\ y\in T\}$
and $K$ will be called an origin-symmetric (or centrally symmetric) convex set if $K=-K$. The convex hull of a set $S$ is the smallest convex set $K$ such that $S\subseteq K$, and is denoted by $\conv\{S\}$. The support function of a convex set $K$ is $h_K(x)=\sup_{y\in K}\iprod{x}{y}$. We denote $[n] = \{1,...,n\}$ for $n \in \N$. 

The study of Gaussian correlation inequalities started with the works of \v{S}id\'ak \cite{sidak1967rectangular,sidak1968multivariate} and Khatri \cite{khatri1967certain}, who independently proved the case of correlation between one random Gaussian and $n-1$ others (or, in the geometric formulation, between a convex body and a slab). It was conjectured by Das Gupta, Eaton, Olkin, Perlman, Savage, and Sobel in \cite{gupta1972inequalities}, and later proved by Royen in \cite{royen2014simple}, that for $K,T \subseteq \RR^d$ centrally symmetric,
\begin{equation}\label{eq:GCI}
    \gamma_d(K\cap T)\ge \gamma_d(K)\gamma_d(T).
\end{equation}
This inequality is known as the Gaussian Correlation Inequality.

The history of this problem is as follows. Pitt \cite{pitt1977gaussian} published a proof of inequality \eqref{eq:GCI} in two dimensions, and introduced the decoupling approach to the problem. However, at the time Pitt's results could not be translated to higher dimension, and so proofs of special cases were given: in \cite{Schechtman1998gaussian} Schechtman, Schlumprecht, and Zinn proved the inequality for two ellipsoids, for a body and a ball, and in some other cases. In \cite{harge1999particular}, Harg{\'e} proved the inequality for a body and an ellipsoid, and in \cite{cordero2002contraction}, Cordero-Erasquin used the Cafarelli contraction theorem to give another proof of this case.

Finally, Royen proved the Gaussian correlation conjecture in \cite{royen2014simple}, using Pitt's decoupling approach together with some fine analysis. He also proved that the only equality cases in \eqref{eq:GCI} are for `orthogonal' sets, i.e. for $K = K_1 \times \RR^{d_2}, T = \RR^{d_1} \times T_2$ in $\RR^{d} = \RR^{d_1 + d_2}$. Royen's proof is actually applicable to to a large family of distributions, for a simplified reading of his proof see \cite{latala2017royen}.

Since Royen's proof, there has been some work on expanding and refining the conjecture. Eskenazis, Nayar and Tkocz \cite{eskenazis2018gaussian} have extended the Gaussian correlation inequality to Gaussian mixtures, thus also proving a spherical correlation inequality (with respect to convex sets on the sphere).

The version of Conjecture \ref{conj:strong-GCI} presented in this note was proposed by Tehranchi in \cite{Tehranchi},
who also strengthened a result of Schechtman, Schlumprecht, and Zinn \cite{Schechtman1998gaussian}, and showed that 
\[ \gamma_d(K)\gamma_d(T) \le (1-s)^{-d/2} \gamma_d \left( \sqrt{\frac{2(1-s)}{1+t}} (K \cap T)\right) \gamma_d \left( \sqrt{\frac{1-s}{2(1-t)}} (K + T)\right)\]
for all $\sqrt{s}\le t< 1$.
Conjecture \ref{conj:strong-GCI} is also motivated by the fact the a similar inequality holds for the Lebesgue measure, namely for $K,T$ origin-symmetric,
\[\vol(K+T)\vol(K\cap T)\ge \vol(K)\vol(T),\]
as shown by Rogers and Shephard in \cite{rogers-shephard} (for more on this see \cite[Section 1.5]{AGA-i}).

A final motivation for this work is the quantitative version of the Gaussian Correlation Inequality. In \cite[Theorem 31]{shivam2020quantitative}, De, Nadimpalli and Servedio defined a function $c(K,T) \geq 0$ and determined that
$$\gamma_d(K\cap T)- \gamma_d(K)\gamma_d(T)\ge c(K,T) .$$
Conjecture \ref{conj:strong-GCI} is another example of a quantitative bound, as it can be rewritten as
\[
    \gamma_d(K \cap T) - \gamma_d(K) \gamma_d(T) \geq \left( \frac{1}{\gamma_d(K+T)} - 1 \right) \gamma_d(K) \gamma_d(T) .
\]

The probabilistic version of the strong Gaussian Correlation conjecture, Conjecture \ref{conj:strong-GCI-prob}, is new, as far as we know. However, it is equivalent to \ref{conj:strong-GCI}:
\begin{prop}\label{prop:equivalent-formulations}
    Conjecture \ref{conj:strong-GCI} holds if and only if Conjecture \ref{conj:strong-GCI-prob} holds.
\end{prop}

\begin{proof}
    Assume that Conjecture~\ref{conj:strong-GCI} holds, and denote by $X=(X_1,\dots, X_n)$ the centered Gaussian vector whose indices are the Gaussian variables of Conjecture \ref{conj:strong-GCI-prob}. The Gaussian vector $X$ can be expressed as a linear image of a standard Gaussian $Y\in \RR^d$ for some $d \in [n]$, i.e. there exist $u_1,\dots, u_n \in \RR^d$ such that $X_i=\iprod{Y}{u_i}$. Construct the following two convex bodies in $\RR^d$,
    \begin{equation} \label{eq:geometric-to-prob}
        K=\{y:  |\iprod{y}{u_i}| \le s_i \ \forall i \in [n]\},\ \ T=\{y:  |\iprod{y}{u_i}| \le t_i \ \forall i \in [n]\}.  
    \end{equation}
    Note that $K+T \subseteq \{y: |\iprod{y}{u_i}| \le s_i+t_i \ \forall i \in [n]\}$, so applying \eqref{eq:strong-GCI}, we get that
    \begin{align*}
         &\Pr \parens*{|X_i|\le s_i+t_i \ \forall i \in [n]} \cdot \Pr \parens*{|X_i|\le \min\{s_i,t_i\} \ \forall i \in [n]} \\
         & \quad \ge \gamma_d(K+T)\gamma_d(K\cap T) \ge \gamma_d(K)\gamma_d(T)\\
         & \quad = \Pr \parens*{|X_i|\le s_i \ \forall i \in [n]} \cdot \Pr \parens*{|X_i|\le t_i \ \forall i \in [n]} .
    \end{align*}
    In the other direction, assume Conjecture~\ref{conj:strong-GCI-prob} holds. Let $K$ and $T$ be origin-symmetric convex bodies in $\RR^d$ for some $d \in [n]$. Assume that $K,T$ are polytopes; the inequality for general origin-symmetric convex bodies follows by a standard approximation argument.

    Let $\{u_i\}_{i=1}^n$ denote the set of outer unit normals of facets of $K+T$. Note that any outer unit normal of a facet of $K$, $T$, or $K \cap T$ is already in this set. Taking $Y$ a $d$-dimensional standard Gaussian, we construct $X_1,\dots, X_n$ jointly Gaussian random variables by taking $X_i=\iprod{Y}{u_i}$. Then applying \eqref{eq:strong-GCI-prob} with $\{X_i\}_{i=1}^n$, $s_i = h_K(u_i)$, and $t_i = h_T(u_i)$, we recall that $h_{K+T}=h_K+h_T$ and $h_{K\cap T} = \min\{h_K, h_T\}$ to get the following:
    \begin{align*}
        &\gamma_d(K+T) \cdot \gamma_d(K \cap T) \\
        & \quad = \Pr \parens*{|X_i|\le s_i+t_i \ \forall i \in [n]} \cdot \Pr \parens*{|X_i|\le \min\{s_i,t_i\} \ \forall i \in [n]}  \\
        & \quad \ge \Pr \parens*{|X_i|\le s_i \ \forall i \in [n]} \cdot \Pr \parens*{|X_i|\le t_i \ \forall i \in [n]}  = \gamma_d(K)\gamma_d(T) .
    \end{align*}
\end{proof}

\begin{rem}\normalfont
    This proof is quite standard, and a similar variant of it has been used to show the equivalence of Equations~\eqref{eq:prob-GCI} and~\eqref{eq:GCI}.
\end{rem}

\section{Proofs}
\subsection{The refined \v{S}id\'ak-Khatri inquality}
To prove Theorem \ref{thm:refined-sidak-prob} we will use a lemma similar to Giannopoulos's proof \cite{giannopoulos1997vector-ballancing} of the \v{S}id\'ak-Khatri inquality:

    \begin{lem}\label{lem:lift}
        Let $d_1, d_2 \in \N$ and $d = d_1 + d_2$. Let $K$ be an origin-symmetric convex body in $\RR^d = \RR^{d_1} \times \RR^{d_2}$ and let $T = T_1 \times \RR^{d_2}$ be a product body, where $T_1$ is a convex body in $\RR^{d_1}$. Denote by $\pi_i: \RR^d \to \RR^{d_i}$, $i=1,2$, the orthogonal projections.
        
        For any $s \in \RR^{d_1}$ define the function $f(s) = \gamma_{d_2}(K \cap \pi_1^{-1}(s))$, and assume that for any $z > 0$, $\setdef{s \in \pi_1(K)}{f(s) \geq z}$ and $T_1$ satisfy \eqref{eq:geometric-sidak} in $\RR^{d_1}$.
        
        Then $K$ and $T$ also satisfy $\eqref{eq:geometric-sidak}$ in $\RR^d$.
    \end{lem}
    We leave the proof of this lemma for later, and show that it immediately gives a proof of Theorem \ref{thm:slab}:
    \begin{proof}[Proof of Theorem \ref{thm:slab}]
        Assume without loss of generality that $u\in \sp\{e_1\} \subset \RR^d$, where $e_1$ is the first standard basis vector. In the notation of Lemma \ref{lem:lift}, take $d_1=1, d_2=d-1$, then $T_1$ is a symmetric interval $[-a, a]$ (with $a=1/\|u\|$) and for any $z > 0$, $\setdef{s \in \pi_1(K)}{f(s) \geq z}$ is also some symmetric interval, $[-b(z),b(z)]$, by log-concavity of $f$. These intervals clearly satisfy \eqref{eq:geometric-sidak}, since for $d=1$, Theorem \ref{thm:slab} holds trivially.
    \end{proof}
    
    Using Theorem \ref{thm:slab} we prove the refined \v{S}id\'ak-Khatri inequality.
    
    \begin{proof}[Proof of Theorem \ref{thm:refined-sidak-prob}]
    Denote $X=(X_1,\dots, X_n)$. The Gaussian vector $X$ can be expressed as a linear image of a standard Gaussian $Y\in \RR^d$ for some $d \in [n]$, i.e. there exist $u_1,\dots, u_n \in \RR^d$ such that $X_i=\iprod{Y}{u_i}$.
    Taking $s_1 = 1+a_1, s_2 =1,\dots, s_n=1$ and $t_1=1, t_2=\infty,\dots, t_n = \infty$, we construct the two bodies
    \[
        K=\{y\in \RR^d:  |\iprod{y}{u_i}| \le s_i \ \forall i \in [n]\},
    \ \ T=\{y \in \RR^d:  |\iprod{y}{u_i}| \le t_i \ \forall i \in [n]\},
    \]
    and note that $T$ is a symmetric slab in direction $u_1$. Applying Theorem \ref{thm:slab}, and using the fact that $\conv(K \cup T) \subseteq \{y: |\iprod{y}{u_1}|\le 1+a_1\}$, we find that
    \begin{align*}
        &\Pr(|X_1|\le 1+a_1) \cdot \Pr(|X_i|\le 1  \ \forall i \in [n])\\
        & \quad = \gamma_d(\conv(K \cup T)) \cdot \gamma_d(K \cap T)\ge \gamma_d(K) \cdot \gamma_d(T)\\
        & \quad = \Pr(|X_1|\le 1) \cdot \Pr(|X_1|\le 1+a_1,\dots, |X_n|\le 1  ).
    \end{align*}
    This proves \eqref{eq:simple-sidak}, and since the choice of the first coordinate $X_1$ was arbitrary, we may apply the inequality inductively on other coordinates and find that for all $a_1,\dots, a_n>0$,
    \[
        \dfrac{\Pr(|X_i|\le 1 \ \forall i \in [n])}{ \prod_{i=1}^n \Pr(|X_i|\le 1) } \ge  \dfrac{\Pr(|X_i|\le 1+a_i \ \forall i \in [n])}{ \prod_{i=1}^n \Pr(|X_i|\le 1+a_i) } .
    \]
    Since the inequality holds for all $\{a_i\}$, we prove \eqref{eq:refined-sidak}.
\end{proof}

\begin{rem}\label{rem:confidence-interval}\normalfont
    Our refined \v{S}id\'ak inequality may be reinterpreted for $Y_1,\dots, Y_n \sim N(0,1)$ standard Gaussians (Taking $Y_i = \frac{X_i}{\sigma_i}$) as 
        \begin{equation*}
            \dfrac{\Pr(|Y_1|\le c_1,\dots, |Y_n|\le c_n  )}{ \prod_{i=1}^n \Pr(|Y_i|\le c_i) } \ge \sup_{a_1,\dots,a_n\in (0,\infty) } \dfrac{\Pr(|Y_1|\le c_1+a_1,\dots, |Y_n|\le c_n+a_n  )}{ \prod_{i=1}^n \Pr(|Y_i|\le c_i+a_i) }, 
    \end{equation*}
    for any choice of $c_i > 0$ for all $i$.
    In \v{S}id\'ak's original paper, he uses the same inequality with 1 on the left hand side to obtain a confidence rectangle for $Y_1,\dots, Y_n$. In particular, for confidence level of $(1-\alpha)$ one takes $c_i=c$ such that $\Phi(c)=\frac{1}{2}(1+(1-\alpha)^{1/k} )$.

    In the refined inequality, if one may find any $a>0$ such that 
    $$ A = \frac{\Pr(|Y_i|\le c+a \ \forall i \in [n])} { \prod_{i=1}^n \Pr(|Y_i|\le c+a)} >1$$
    then the same choice of $c_i=c$ would improve the confidence level to $A(1-\alpha)$.
\end{rem}

Finally, let us return to the proof of the lemma from the beginning of this subsection.
    \begin{proof}[Proof of Lemma \ref{lem:lift}]
        Note that $\conv\{K \cup T\} = \pi_1(\conv\{K \cup T\}) \times \RR^{d_2}$. Then by Fubini's theorem,
        \begin{align*}
            \gamma_d(\conv\{ K\cup T\} ) \gamma_d( K \cap T )  &= \gamma_{d_1}(\pi_1(\conv\{ K \cup T\})) \int_{\RR^{d_1}} f(s) \mathbbm{1}_{T_1}(s) \diff \gamma_{d_1}(s)  \\
            &= \gamma_{d_1}(\conv\{ \pi_1(K) \cup T_1\})  \int_0^\infty \gamma_{d_1}(\setdef{s \in T_1}{f(s) \geq z}) \diff z \\
            &= \int_0^\infty \gamma_{d_1}(\{f \geq z\} \cap T_1) \cdot \gamma_{d_1}(\conv\{ \pi_1(K) \cup T_1\}) \diff z \\
            &\geq \int_0^\infty \gamma_{d_1}(\{f \geq z\} \cap T_1) \cdot \gamma_{d_1}(\conv\{ \{f \geq z\} \cup T_1\}) \diff z ,
        \end{align*}
        where the last inequality holds since for any $z > 0$, $\{f \geq z\} \subseteq \pi_1(K)$ by definition. By the assumption, for any $z > 0$,
        \begin{gather*}
            \gamma_{d_1}(\{f \geq z\} \cap T_1) \cdot \gamma_{d_1}(\conv\{ \{f \geq z\} \cup T_1\} ) \geq \gamma_{d_1}(\{f \geq z\}) \cdot \gamma_{d_1}(T_1) .
        \end{gather*}
        Therefore,
        \begin{align*}
            \gamma_d(\conv\{ K\cup T\} ) \gamma_d( K \cap T )  \geq \int_0^\infty \gamma_{d_1}(\{f \geq z\}) \diff z \cdot \gamma_{d_1}(T_1) = \gamma_d(K) \cdot \gamma_{d_1}(T_1) .
        \end{align*}
        Noting that $\gamma_d(T) = \gamma_{d_1}(T_1)$, the proof is completed.
    \end{proof}

    \begin{rems}\normalfont
    \begin{itemize}
        \item The only property of $\gamma_d$ used in the proof is that it is a product  probability measure with respect to the splitting $\RR^d = \RR^{d_1} \times \RR^{d_2}$.
        \item If the assumption that $\setdef{s \in \pi_1(K)}{f(s) \geq z}$ and $T_1$ satisfy \eqref{eq:geometric-sidak} in $\RR^{d_1}$ is replaced with the weaker assumption that the sets satisfy \eqref{eq:strong-GCI}, then we can conclude, with the same proof, that $K$ and $T$ satisfy \eqref{eq:strong-GCI} in $\RR^d$.
    \end{itemize}
        
    \end{rems}

\subsection{Stronger Gaussian correlation inequality}
We begin with a simple proof of Theorem \ref{thm:uncond-strong-GCI}, which shows that \eqref{eq:strong-GCI} holds for unconditional convex sets. Recall that a convex set is called unconditional if it is symmetric with respect to all coordinate hyperplanes. Following \cite{Schechtman1998gaussian}, we use the Karlin-Rinott Theorem. Denote by $\preccurlyeq$ the product partial order on $\RR^n$, i.e. $x \preccurlyeq y$ if and only if $\forall i: x_i \leq y_i$. This is a lattice; denote by $\lor, \land$ join and meet operations (i.e. coordinate-wise maximum and minimum).
\begin{thm}[{\cite[Theorem 2.1]{karlin1980classes}}]\label{thm:KR}
    Let $\mu$ be a product measure on $\RR^d$ and let $f_i$, $1 \le i \le 4$,
be nonnegative functions on $\RR^d$ satisfying
$$f_1(x)f_2(y) \le f_3(x \lor y)f_4(x \land y).$$
Then,
\[\int f_1 d\mu \int f_2 d\mu \le \int f_3 d\mu \int f_4 d\mu.\]
\end{thm}
We mention that this theorem is a continuous version of the celebrated four-function theorem \cite{four-function}.

\begin{proof}[Proof of Theorem \ref{thm:uncond-strong-GCI}]
    Let $Q = \{x \in \RR^d \ | \ 0 \preccurlyeq x\}$ be the positive orthant. Note that by uncoditionality of $K$, we have that if $0 \preccurlyeq x \preccurlyeq y \in K$, then also $x \in K$, and similarly for $T$. 
    
    By unconditionality of $K$ and $T$, if $x\in K \cap Q$ and $y\in T \cap Q$, then since $x \land y \preccurlyeq x, y$ we have $x \land y \in K \cap T \cap Q$. By the positivity of elements in $Q$ we have $x \lor y \preccurlyeq x+y$ which implies that $x \lor y \in (K+T) \cap Q$. 
    It follows that $\gamma_d$ and the functions
    \[f_1 = 1_{K\cap Q},\ f_2 = 1_{ T\cap Q},\ f_3=1_{(K+T)\cap Q},\ f_4=1_{K\cap T\cap Q} \]
    satisfy the  assumptions of Theorem \ref{thm:KR}, and we find that
    \[\gamma_d(K\cap Q)\gamma_d(T\cap Q)\le \gamma_d(K\cap T\cap Q)\gamma_d((K+T)\cap Q). \]
    Since $\gamma_d(K)=2^n \gamma_d(K\cap Q)$, and similarly for $T, K\cap T$ and $K+T$, the theorem follows.
\end{proof}

\begin{rem}\normalfont
    Theorem \ref{thm:uncond-strong-GCI} holds for any symmetric product measure, and the proof is verbatim. 
\end{rem}

Next, we see that unlike in Theorem \ref{thm:slab}, in the unconditional case the Minkowski sum in Conjecture \ref{conj:strong-GCI} cannot be replaced with a smaller set such as the convex hull of $K,T$.
\begin{exm}\label{ex:not-for-conv}
Let $K,T\subseteq \RR^2$ and $N\in \RR_+$, with $K=[-1/N,1/N]\times [-N,N]$, and $T=[-N,N]\times[-1/N,1/N]$. Denoting $L = \conv\{\pm e_1, \pm e_2\}$ where $\{e_1,e_2\}$ is the standard orthonormal basis we have
\begin{align*}
 \gamma_2(K)\gamma_2(T)&= \gamma_1([-1/N,1/N])^2\gamma_1([-N,N])^2,
\end{align*}
while
\begin{align*}
    \gamma_2(K\cap T) \gamma_2(\conv (K\cup T)) &= \gamma_1([-1/N,1/N])^2 \gamma_2(\conv (K\cup T)) \\
 & \le \gamma_1([-1/N,1/N])^2 \gamma_2((N+1/N)L ) \\
&= \gamma_1([-1/N,1/N])^2 \gamma_1 \left( \left[-\frac{(N+ 1/N)}{\sqrt{2}},\frac{(N+ 1/N)}{\sqrt{2}}\right]\right)^2
\end{align*}
where the inequality is due to the inclusion $K,T\subseteq (N+1/N)L$, and the volume of the square $(N+1/N)L$ can be computed by the rotation invariance of $\gamma_2$. The inequality 
\[\gamma_2(K)\gamma_2(T)\le \gamma_2(K\cup T) \gamma_2(\conv (K\cup T)) \]
would hence imply that 
\[\gamma_1([-N,N])\le \gamma_1 \left( \left[-\frac{(N+ 1/N)}{\sqrt{2}},\frac{(N+ 1/N)}{\sqrt{2}}\right]\right)\]
which is false for $N \geq 3$.

\end{exm}

We end this subsection with the proof of Theorem \ref{thm:asymptotic-GCI}, which is a standard tensorization argument.

\begin{proof}[Proof of Theorem \ref{thm:asymptotic-GCI}]
    Let $K,T \subseteq \RR^d$ be origin-symmetric sets.
    Denote by $K^N = \prod_{i=1}^N K \subseteq (\RR^d)^N = \RR^{Nd}$, and similarly $T^N = \prod_{i=1}^N T \subseteq (\RR^d)^N = \RR^{Nd}$. Then, by the assumption, for all $N\in \mathbb{N}$,
    \[\gamma_{Nd}(K^N+T^N)\gamma_{Nd}(K^N\cap T^N)\ge c_{Nd} \cdot \gamma_{Nd}(K^N)\gamma_{Nd}(T^N).\]
   Note that $K^N + T^N = (K+T)^N$, and $K^N \cap T^N = (K \cap T)^N$. Since $\gamma_{Nd}$ is a product measure, $\gamma_{Nd}(K^N)= (\gamma_d(K))^N$, $\gamma_{Nd}(T^N)= (\gamma_d(T))^N$, $\gamma_{Nd}(K^N+T^N) = (\gamma_d(K+T))^N$ and  $\gamma_{Nd}(K^N\cap T^N) = (\gamma_d(K \cap T))^N$ . Plugging these into the inequality above and taking $N$-th root gives 
   \[\gamma_d (K+T) \gamma_d (K\cap T) \ge c_{Nd}^{1/N} \gamma_d(K)\gamma_d(T).\]
   As $d$ is constant we may apply our assumption on $c_N$  and take $N \to \infty$ in the inequality above to conclude the proof.
\end{proof}

\bibliographystyle{amsplain}
\addcontentsline{toc}{section}{\refname}\bibliography{GCI}
\end{document}